\documentclass[11pt]{amsart}
\usepackage{amsmath}
\usepackage[active]{srcltx}
\usepackage{t1enc}
\usepackage[latin2]{inputenc}
\usepackage{verbatim}
\usepackage{amsmath,amsfonts,amssymb,amsthm}
\usepackage[mathcal]{eucal}
\usepackage{enumerate}
\usepackage[centertags]{amsmath}
\usepackage{graphics}

\newtheorem{theorem}{Theorem}

\newtheorem{lemma}{Lemma}

\newtheorem{definition}{Definition}
\newtheorem{corollary}{Corollary}

\newtheorem*{OW1}{Theorem OW1}
\newtheorem*{OW2}{Theorem OW2}
\newtheorem*{OW3}{Theorem OW3}

\begin{document}
\author{Gvantsa Shavardenidze}
\address{G.Shavardenidze, Institute of Mathematics, Faculty of Exact and
Natural Sciences, Ivane Javakhishvili Tbilisi State University,
Chavcha\-vadze str. 1, Tbilisi 0179, Georgia}
\email{shavardenidzegvantsa@gmail.com}
\address{ }
\email{ }
\title[]{On the convergence of Cesáro means of negative order of
Vilenkin-Fourier series}
\date{}
\maketitle

\begin{abstract}
In 1971 Onnewer and Waterman establish sufficient condition which guarantees
uniform convergence of Vilenkin-Fourier series of continuous function. In
the paper we consider different classes of functions of generalized bounded
oscilation and in the terms of these classes there are established
sufficient conditions for uniform convergence of Cesáro means of negative
order.
\end{abstract}

\footnotetext{%
}

\section{Definition an Notations}

Let $\mathbb{N}_{+}$ denote the set of positive integers, $\mathbb{N}:=%
\mathbb{N}_{+}\cup \{0\}.$ Let $m:=(m_{0},m_{1},...)$ denote a sequence of
positive integers not less than $2.$ Denote by $Z_{m_{k}}:=\{0,1,...,m_{k}-1%
\}$ the additive group of integers modulo $m_{k}$. Define the group $G_{m}$
as the complete direct product of the groups $Z_{m_{j}},$ with the product
of the discrete topologies of $Z_{m_{j}}$'s. The direct product $\mu $ of
the measures 
\begin{equation*}
\mu _{k}(\{j\}):=\frac{1}{m_{k}}\quad (j\in Z_{m_{k}})
\end{equation*}%
is the Haar measure on $G_{m}$ with $\mu (G_{m})=1.$ If the sequence $m$ is
bounded, then $G_{m}$ is called a bounded Vilenkin group. In this paper we
consider only bounded Vilenkin group. The elements of $G_{m}$ can be
represented by sequences $x:=(x_{0},x_{1},...,x_{j},...)$, $(x_{j}\in
Z_{m_{j}}).$ The group operation $+$ in $G_{m}$ is given by $x+y=\left(
x_{0}+y_{0}\left( \text{mod}m_{0}\right) ,...,x_{k}+y_{k}\left( \text{mod}%
m_{k}\right) ,...\right) $ , where $x=\left( x_{0},...,x_{k},...\right) $
and $y=\left( y_{0},...,y_{k},...\right) \in G_{m}$. The inverse of $+$ will
be denoted by $-$.

It is easy to give a base for the neighborhoods of $G_{m}:$ 
\begin{equation*}
I_{0}(x):=G_{m},
\end{equation*}%
\begin{equation*}
I_{n}(x):=\{y\in G_{m}|y_{0}=x_{0},...,y_{n-1}=x_{n-1}\}
\end{equation*}%
for $x\in G_{m},\ n\in {\mathbb{N}}$. Define $I_{n}:=I_{n}(0)$ for $n\in {%
\mathbb{N}}_{+}$. Set $e_{n}:=\left( 0,...,0,1,0,...\right) \in G_{m}$ the $%
n\,$th\thinspace coordinate of which is 1 and the rest are zeros $\left(
n\in \mathbb{N}\right) .$

If we define the so-called generalized number system based on $m$ in the
following way: $M_{0}:=1,M_{k+1}:=m_{k}M_{k}(k\in {\mathbb{N}}),$ then every 
$n\in {\mathbb{N}}$ can be uniquely expressed as $n=\sum\limits_{j=0}^{%
\infty }n_{j}M_{j},$ where $n_{j}\in Z_{m_{j}}\ (j\in {\mathbb{N}}_{+})$ and
only a finite number of $n_{j}$'s differ from zero. We use the following
notation.

Set $n^{\left( A\right) }=n_{A}M_{A}+\cdots +n_{0}M_{0},$ $n_{A}\neq 0,$ $\
A\in \mathbb{N}.$ \ Suppose that $n^{\left( -1\right) }=0.$

Let $Z_{\beta }^{\left( k\right) }=\left( x_{0},x_{1},\cdots
,x_{k-1},0,0,\cdots \right) ,$ where 
\begin{equation*}
\beta =\sum\limits_{j=1}^{k-1}\left( \frac{x_{j}}{M_{j+1}}\right) M_{k}\text{%
, \ }\left( x_{j}\in Z_{m_{j}}\right) \text{, \ }j=0,1,\ldots ,n-1.
\end{equation*}

It's easy to show that 
\begin{equation}
G_{m}=\bigcup\limits_{\beta =0}^{M_{k}-1}\left( I_{k}+Z_{\beta }^{\left(
k\right) }\right) .  \label{Gm}
\end{equation}

Next, we introduce on $G_{m}$ an orthonormal system which is called the
Vilenkin system \cite{AVDR}. At first define the complex valued functions $%
r_{k}(x):G_{m}\rightarrow {\mathbb{C}}$, the generalized Rademacher
functions in this way 
\begin{equation*}
r_{k}(x):=\exp \frac{2\pi \imath x_{k}}{m_{k}}\ (\imath ^{2}=-1,\ x\in
G_{m},\ k\in \mathbb{N}).
\end{equation*}

\noindent Now define the Vilenkin system $\psi := (\psi_n : n\in{\mathbb{N}}%
) $ on $G_{m}$ as follows. 
\begin{equation*}
\psi_{n}(x):=\prod\limits_{k=0}^{\infty}r_{k}^{n_{k}}(x)\quad (n\in\mathbb{N}%
).
\end{equation*}

\noindent Specifically, we call this system the Walsh-Paley one if $m\equiv
2.$

\noindent The Vilenkin system is orthonormal and complete in $L^{1}(G_{m})$ 
\cite{AVDR}.

\noindent Now, introduce analogues of the usual definitions of the Fourier
analysis. If $f\in L^1(G_m)$ we can establish the following definitions in
the usual way:

\noindent Fourier coefficients: 
\begin{equation*}
\widehat{f}(k):=\int_{G_{m}}f\overline{\psi }_{k}d\mu \qquad (k\in {\mathbb{N%
}}),
\end{equation*}%
partial sums: 
\begin{equation*}
S_{n}f:=\sum_{k=0}^{n-1}\widehat{f}(k)\psi _{k}\qquad (n\in {\mathbb{N}}%
_{+},\,\,S_{0}f:=0),
\end{equation*}%
Fejér means: 
\begin{equation*}
\sigma _{n}f:=\frac{1}{n}\sum_{k=1}^{n}S_{k}f\qquad (n\in {\mathbb{N}}_{+}),
\end{equation*}%
Dirichlet kernels: 
\begin{equation*}
D_{n}:=\sum_{k=0}^{n-1}\psi _{k}\qquad (n\in {\mathbb{N}}_{+}),\text{ \ \ \
\ \ }D_{0}=0,
\end{equation*}%
Fejér kernels: 
\begin{equation*}
K_{n}\left( x\right) :=\frac{1}{n}\sum\limits_{k=1}^{n}D_{k}\left( x\right) .
\end{equation*}

Recall that \cite{SWSP} 
\begin{equation}
D_{M_{n}}\left( x\right) =\left\{ 
\begin{array}{ll}
M_{n}, & \mbox{if }x\in I_{n}, \\ 
0, & \mbox{if }x\in G_{m}\backslash I_{n}.%
\end{array}%
\right.  \label{dir}
\end{equation}

and 
\begin{equation}
\int\limits_{G_{m}}D_{n}\left( t\right) d\mu t=1,n\in N_{+}.  \label{dir1}
\end{equation}

Let $n=n_{k}M_{k}+n^{^{\prime }}$, $0<n_{k}<m_{k}$ and $0\leq n^{^{\prime
}}<M_{k}$, then (see \cite{SWSP}) 
\begin{equation}
D_{n}\left( x\right) =\frac{1-\psi _{M_{k}}^{n_{k}}\left( x\right) }{1-\psi
_{M_{k}}\left( x\right) }D_{M_{k}}\left( x\right) +\psi
_{M_{k}}^{n_{k}}\left( x\right) D_{n^{^{\prime }}}\left( x\right)
\label{dir2}
\end{equation}%
\begin{equation}
D_{j+n_{k}M_{k}}\left( x\right) =D_{n_{k}M_{k}}\left( x\right) +\psi
_{n_{k}M_{k}}\left( x\right) D_{j}\left( x\right) .  \label{dir3}
\end{equation}%
\begin{equation}
D_{j+rM_{k}}\left( x\right) =\left( \sum\limits_{q=0}^{r-1}\psi
_{M_{k}}^{q}\left( x\right) \right) D_{M_{k}}\left( x\right) +\psi
_{M_{k}}^{r}\left( x\right) D_{j}\left( x\right) .  \label{dir4}
\end{equation}%
Let $0\leq j<n_{s}M_{s}$ and $0\leq n_{s}<M_{s}$, then ( see \cite{GatGog}) 
\begin{equation}
D_{n_{s}M_{s}-j}\left( x\right) =D_{n_{s}M_{s}}\left( x\right) +\psi
_{n_{s}M_{s}-1}\left( x\right) \overline{D_{j}\left( x\right) }.
\label{dir5}
\end{equation}

\bigskip Cesáro $\left( C,\alpha \right) $ means of Vilenkin-Fourier series
is defined as follows%
\begin{eqnarray*}
\sigma _{n}^{-\alpha }\left( f;x\right) &=&\frac{1}{A_{n-1}^{-\alpha }}%
\sum\limits_{\nu =0}^{n}A_{n-\nu }^{-\alpha -1}S_{\nu }\left( f;x\right) \\
&=&\frac{1}{A_{n-1}^{-\alpha }}\sum\limits_{\nu =0}^{n-1}A_{n-\nu }^{-\alpha
}\widehat{f}\left( \nu \right) \psi _{\nu }\left( x\right) \\
&=&\int\limits_{G_{m}}f\left( x-t\right) K_{n}^{-\alpha }\left( t\right)
d\mu \left( t\right) .
\end{eqnarray*}%
where%
\begin{equation*}
A_{0}^{\alpha }=0,
\end{equation*}%
\begin{equation*}
A_{n}^{\alpha }=\frac{\left( \alpha +1\right) \cdots \left( \alpha +n\right) 
}{n!},\alpha \neq -1,-2,\ldots ,
\end{equation*}%
\begin{equation*}
K_{n}^{-\alpha }\left( t\right) =\frac{1}{A_{n-1}^{-\alpha }}%
\sum\limits_{\nu =0}^{n-1}A_{n-1}^{-\alpha }\psi _{\nu }\left( t\right) .
\end{equation*}%
It is well known that \cite{Zy}%
\begin{equation}
A_{n}^{\alpha }=\sum\limits_{k=0}^{n-1}A_{n-k}^{\alpha -1},  \label{A1}
\end{equation}%
\begin{equation}
A_{n}^{\alpha }\sim n^{\alpha },  \label{A2}
\end{equation}%
\begin{equation}
A_{n}^{\alpha }-A_{n-1}^{\alpha }=A_{n}^{\alpha -1}.  \label{A3}
\end{equation}

By $C\left( G_{m}\right) $ denote the space of continuous functions on $%
G_{m} $ with the supremum norm%
\begin{equation*}
\left\Vert f\right\Vert _{C}:=\sup_{x\in G_{m}}\left\vert f\left( x\right)
\right\vert ,\ \left( f\in G_{m}\right) .
\end{equation*}%
Let $f\in C\left( G_{m}\right) .$ The modulus of continuity is defined as
follows%
\begin{equation*}
\omega \left( f,\frac{1}{M_{k}}\right) :=\sup_{x\in G_{m}}\sup_{t\in
I_{k}}\left\vert f\left( x-t\right) -f\left( x\right) \right\vert .
\end{equation*}%
Set 
\begin{equation*}
O\left( f,M_{k}\right) =\sum\limits_{\beta =1}^{M_{k}-1}\omega \left(
f,I_{k}+Z_{\beta }^{\left( k\right) }\right) ,
\end{equation*}%
where%
\begin{equation*}
\omega \left( f,I_{k}+Z_{\beta }^{\left( k\right) }\right)
=\sup_{x,x^{\prime }\in I_{k}+Z_{\beta }^{\left( k\right) }}\left\vert
f\left( x\right) -f\left( x^{^{\prime }}\right) \right\vert .
\end{equation*}

\begin{definition}
\cite{SWSP} We say that $f$ is function of Bounded oscilation $\left( f\in
BO\left( G_{m}\right) \right) ,$ if%
\begin{equation*}
\sup_{k}O\left( f,M_{k}\right) <\infty .
\end{equation*}
\end{definition}

For bounded Vilenkin group Onnewer and Waterman \cite{OW} proved that the
following theorems hold true.

\begin{OW1}[Onnewer,Waterman]
\label{theoremOW1} Let $f$ be a continuous functions on $G_{m}$ such that 
\begin{equation*}
\lim_{k\rightarrow \infty }\sum\limits_{\beta =1}^{M_{k}-1}\frac{1}{\beta }%
\left\vert \sum\limits_{j=0}^{m_{k+1}-1}f\left( x-Z_{\beta }^{\left(
k\right) }-jx_{k}\right) e^{\frac{2\pi ijn_{k}}{m_{k+1}}}\right\vert =0
\end{equation*}%
uniformly in $x\in G_{m}$ and $n_{k}\in \left\{ 1,2,\ldots
,m_{k+1}-1\right\} .$ Then the Vilenkin-Fourier series of $f$ converges
uniformly on $G_{m}.$
\end{OW1}

From Theorem \ref{theoremOW1} imply the following theorem.

\begin{OW2}
\label{theoremOW2} \cite{OW} Let $f\in C\left( G_{m}\right) \cap BO\left(
G_{m}\right) .$ Then the Vilenkin-Fourier series of $f$ converges uniformly
on $G_{m}.$
\end{OW2}

Let $p\left( u\right) $ be a continuous, realvalued, strictly increasing
function, defined for $u\geq 0$, such that $p\left( 0\right) =0$ and $%
\lim_{u\rightarrow \infty }p\left( u\right) =\infty .$ Let $q\left( u\right) 
$ be the inverse of $p\left( u\right) .$ Let%
\begin{equation*}
M\left( u\right) =\int\limits_{0}^{u}p\left( t\right) dt\text{ \ \ \ \ and \
\ \ \ \ }N\left( u\right) =\int\limits_{0}^{u}q\left( t\right) dt\text{ .\ }
\end{equation*}%
Functions $M$ and $N$ thus obtained are called complementary in the sense of
Young \cite{Zy}, and they satisfy the following inequality%
\begin{equation*}
\text{if \ }a,b\geq 0,\text{ then }ab\leq M\left( a\right) +N\left( b\right)
.
\end{equation*}

\begin{definition}
A function $f$ on $G_{m\text{ }}$ is generalized bounded M-oscilation \ $%
\left( f\in BO_{M}\left( G_{m}\right) \right) $ if there exists an $K<\infty 
$, such that 
\begin{equation*}
\sup_{k}\sum\limits_{\beta =1}^{M_{k}-1}M\left( \omega \left(
f,I_{k}+Z_{\beta }^{\left( k\right) }\right) \right) <K.
\end{equation*}
\end{definition}

In terms of M-oscilation Onnewer and Waterman \cite{OW} proved that the
following is true.

\begin{OW3}
Let $M$ and $N$ be functions complementary in the sense of Young, $f\in
C\left( G_{m}\right) \cap BO_{M}\left( G_{m}\right) $ and let $%
\sum\limits_{k=1}^{\infty }N\left( k^{-1}\right) <\infty .$ Then the Fourier
series of $f$ converges uniformly on $G_{m}.$
\end{OW3}

\section{Main Results}

\begin{theorem}
\label{theorem1}Let $f\in C\left( G_{m}\right) ,\alpha \in \left( 0,1\right) 
$ and 
\begin{equation*}
\lim_{k\rightarrow \infty }\sum\limits_{\beta =1}^{M_{k}-1}\frac{1}{\beta
^{1-\alpha }}\left\vert f\left( x-z_{\beta }^{\left( k\right) }\right)
-f\left( x-z_{\beta }^{\left( k\right) }-e_{k}\right) \right\vert =0
\end{equation*}

uniformly with respect to $x$ on $G_{m}$. Then%
\begin{equation*}
\lim_{n\rightarrow \infty }\left\Vert \sigma _{n}^{-\alpha }\left( f\right)
-f\right\Vert _{C}=0.
\end{equation*}
\end{theorem}

Set

\begin{equation*}
\nu \left( M_{k},f\right) =\sum\limits_{\beta =0}^{M_{k}-1}\omega \left(
f,I_{k}+Z_{\beta }^{\left( k\right) }\right) .
\end{equation*}

\begin{theorem}
\label{theorem2} Let $f\in C\left( G_{m}\right) ,\alpha \in \left(
0,1\right) .$ If%
\begin{equation*}
\sum\limits_{k=1}^{\infty }\frac{\nu \left( M_{k},f\right) }{M_{k}^{1-\alpha
}}<\infty ,
\end{equation*}%
then%
\begin{equation*}
\lim_{n\rightarrow \infty }\left\Vert \sigma _{n}^{-\alpha }\left( f\right)
-f\right\Vert _{C}=0.
\end{equation*}
\end{theorem}

From this theorem we get following result.

\begin{corollary}
\label{shedegi1}Let $f\in C\left( G_{m}\right) \cap BO_{M}\left(
G_{m}\right) ,\alpha \in \left( 0,1\right) $ and%
\begin{equation*}
\sum\limits_{k=1}^{\infty }M_{k}^{\alpha }M^{-1}\left( \frac{1}{M_{k}}%
\right) <\infty ,
\end{equation*}%
then%
\begin{equation*}
\lim_{n\rightarrow \infty }\left\Vert \sigma _{n}^{-\alpha }\left( f\right)
-f\right\Vert _{C}=0.
\end{equation*}
\end{corollary}

\begin{corollary}
\label{shedegi2}\bigskip Let $f\in C\left( G_{m}\right) \cap BO_{p}\left(
G_{m}\right) ,0<\alpha <\frac{1}{p},$ then%
\begin{equation*}
\lim_{n\rightarrow \infty }\left\Vert \sigma _{n}^{-\alpha }\left( f\right)
-f\right\Vert _{C}=0.
\end{equation*}
\end{corollary}

We note that the problems of summability of Cesáro means of negative order
of Fourier series with respect to trigonometric, Walsh-Paley and
Walsh-Kaczmarz systems were studied by Zygmund \cite{Zy}, Zhizhiashvili \cite%
{Zhi}, Nagy \cite{nagy1}-\cite{nagy3}, Tevzadze \cite{Tev}, Goginava \cite%
{gog} - \cite{gog1}.

\section{Auxiliary Results}

\begin{lemma}
\label{Lemma1}Let $\alpha \in \left( 0,1\right) $ and $M_{A}\leq n<M_{A+1}.$
Then%
\begin{eqnarray*}
\sum\limits_{j=1}^{n}A_{n-j}^{-\alpha -1}D_{j}\left( x\right)
&=&\sum\limits_{k=0}^{A}\left( \prod\limits_{l=k+1}^{A}\psi
_{n_{l}M_{l}}\left( x\right) \right) D_{n_{k}M_{k}}\left( x\right)
A_{n^{\left( k\right) }-1}^{-\alpha } \\
&&-\sum\limits_{k=0}^{A}\left( \prod\limits_{l=k+1}^{A}\psi
_{n_{l}M_{l}}\left( x\right) \right) \sum\limits_{j=0}^{n_{k}M_{k}-1}\psi
_{n_{k}M_{k}-1}\left( x\right) A_{n^{\left( k-1\right) }+j}^{-\alpha -1}%
\overline{D_{j}\left( x\right) }.
\end{eqnarray*}
\end{lemma}

\begin{proof}
Let $n=n_{A}M_{A}+n^{^{\prime }},$ $0\leq n^{^{\prime }}<M_{A}.$Using (\ref%
{dir3}) and (\ref{A1}) we can write 
\begin{eqnarray*}
\sum\limits_{j=1}^{n}A_{n-j}^{-\alpha -1}D_{j}\left( x\right)
&=&\sum\limits_{j=1}^{n_{A}M_{A}}A_{n-j}^{-\alpha -1}D_{j}\left( x\right)
+\sum\limits_{j=n_{A}M_{A}+1}^{n}A_{n-j}^{-\alpha -1}D_{j}\left( x\right) \\
&=&\sum\limits_{j=1}^{n_{A}M_{A}}A_{n-j}^{-\alpha -1}D_{j}\left( x\right)
+\sum\limits_{j=1}^{n^{\left( A-1\right) }}A_{n^{\left( A-1\right)
}-j}^{-\alpha -1}D_{n_{A}M_{A}+j}\left( x\right) \\
&=&\sum\limits_{j=1}^{n_{A}M_{A}}A_{n-j}^{-\alpha -1}D_{j}\left( x\right)
+D_{n_{A}M_{A}}\left( x\right) A_{n^{\left( A-1\right) }-1}^{-\alpha } \\
&&+\psi _{n_{A}M_{A}}\left( x\right) \sum\limits_{j=1}^{n^{\left( A-1\right)
}}A_{n^{\left( A-1\right) }-j}^{-\alpha -1}D_{j}\left( x\right) .
\end{eqnarray*}%
Using iteration, we get%
\begin{eqnarray}
\sum\limits_{j=1}^{n}A_{n-j}^{-\alpha -1}D_{j}\left( x\right)
&=&\sum\limits_{k=0}^{A}\left( \prod\limits_{l=k+1}^{A}\psi
_{n_{l}M_{l}}\left( x\right) \right)
\sum\limits_{j=1}^{n_{k}M_{k}}A_{n^{\left( k\right) }-j}^{-\alpha
-1}D_{j}\left( x\right)  \label{L1} \\
&&+\sum\limits_{k=0}^{A}\left( \prod\limits_{l=k+1}^{A}\psi
_{n_{l}M_{l}}\left( x\right) \right) D_{n_{k}M_{k}}\left( x\right)
A_{n^{\left( k-1\right) }-1}^{-\alpha }.  \notag
\end{eqnarray}

From (\ref{dir5}) we have%
\begin{eqnarray}
&&\sum\limits_{k=0}^{A}\left( \prod\limits_{l=k+1}^{A}\psi
_{n_{l}M_{l}}\left( x\right) \right)
\sum\limits_{j=1}^{n_{k}M_{k}}A_{n^{\left( k\right) }-j}^{-\alpha
-1}D_{j}\left( x\right)  \label{22} \\
&=&\sum\limits_{k=0}^{A}\left( \prod\limits_{l=k+1}^{A}\psi
_{n_{l}M_{l}}\left( x\right) \right)
\sum\limits_{j=0}^{n_{k}M_{k}-1}A_{n^{\left( k-1\right) }+j}^{-\alpha
-1}D_{n_{k}M_{k}-j}\left( x\right)  \notag \\
&=&\sum\limits_{k=0}^{A}\left( \prod\limits_{l=k+1}^{A}\psi
_{n_{l}M_{l}}\left( x\right) \right) D_{n_{k}M_{k}}\left( x\right)
\sum\limits_{j=0}^{n_{k}M_{k}-1}A_{n^{\left( k-1\right) }+j}^{-\alpha -1} 
\notag \\
&&-\sum\limits_{k=0}^{A}\left( \prod\limits_{l=k+1}^{A}\psi
_{n_{l}M_{l}}\left( x\right) \right) \psi _{n_{k}M_{k}-1}\left( x\right)
\sum\limits_{j=0}^{n_{k}M_{k}-1}A_{n^{\left( k-1\right) }+j}^{-\alpha -1}%
\overline{D_{j}\left( x\right) }.  \notag
\end{eqnarray}

Since 
\begin{equation}
\sum\limits_{j=0}^{n_{k}M_{k}-1}A_{n^{\left( k-1\right) }+j}^{-\alpha
-1}=A_{n^{\left( k\right) }-1}^{-\alpha }-A_{n^{\left( k-1\right)
}-1}^{-\alpha }  \label{L3}
\end{equation}

From (\ref{L1})- (\ref{L3}) we obtain%
\begin{equation*}
\sum\limits_{j=1}^{n}A_{n-j}^{-\alpha -1}D_{j}\left( x\right)
=\sum\limits_{k=1}^{A}\left( \prod\limits_{l=k+1}^{A}\psi
_{n_{l}M_{l}}\left( x\right) \right) D_{n_{k}M_{k}}\left( x\right) \left(
A_{n^{\left( k\right) }-1}^{-\alpha }-A_{n^{\left( k-1\right) }-1}^{-\alpha
}\right)
\end{equation*}%
\begin{equation*}
+\left( \prod\limits_{l=1}^{A}\psi _{n_{l}M_{l}}\left( x\right) \right)
D_{n_{0}M_{0}}\left( x\right) \sum\limits_{j=0}^{n_{0}M_{0}-1}A_{n^{\left(
k-1\right) }+j}^{-\alpha -1}
\end{equation*}%
\begin{equation*}
-\sum\limits_{k=0}^{A}\left( \prod\limits_{l=k+1}^{A}\psi
_{n_{l}M_{l}}\left( x\right) \right) \psi _{n_{k}M_{k}-1}\left( x\right)
\sum\limits_{j=0}^{n_{k}M_{k}-1}A_{n^{\left( k-1\right) }+j}^{-\alpha -1}%
\overline{D_{j}\left( x\right) }
\end{equation*}%
\begin{equation*}
+\sum\limits_{k=0}^{A}\left( \prod\limits_{l=k+1}^{A}\psi
_{n_{l}M_{l}}\left( x\right) \right) D_{n_{k}M_{k}}\left( x\right)
A_{n^{\left( k-1\right) }-1}^{-\alpha }
\end{equation*}%
\begin{equation*}
=\sum\limits_{k=1}^{A}\left( \prod\limits_{l=k+1}^{A}\psi
_{n_{l}M_{l}}\left( x\right) \right) D_{n_{k}M_{k}}\left( x\right)
A_{n^{\left( k\right) }-1}^{-\alpha }
\end{equation*}%
\begin{equation*}
-\sum\limits_{k=0}^{A}\left( \prod\limits_{l=k+1}^{A}\psi
_{n_{l}M_{l}}\left( x\right) \right) \psi _{n_{k}M_{k}-1}\left( x\right)
\sum\limits_{j=0}^{n_{k}M_{k}-1}A_{n^{\left( k-1\right) }+j}^{-\alpha -1}%
\overline{D_{j}\left( x\right) }.
\end{equation*}

Lemma \ref{Lemma1} is proved.
\end{proof}

\begin{lemma}
\label{Lemma2}Let $\alpha \in \left( 0,1\right) $ and $M_{A}\leq n<M_{A+1}$.
Then%
\begin{equation*}
\left\vert K_{n}^{-\alpha }\left( x\right) \right\vert \leq \frac{c\left(
\alpha \right) }{A_{n-1}^{-\alpha }}\sum\limits_{l=0}^{A}M_{l}^{-\alpha
}D_{M_{l}}\left( x\right) .
\end{equation*}
\end{lemma}

\begin{proof}
\bigskip From Lemma \ref{Lemma1} we can write%
\begin{equation*}
\left\vert \sum\limits_{j=1}^{n}A_{n-j}^{-\alpha -1}D_{j}\left( x\right)
\right\vert \leq \sum\limits_{k=0}^{A}D_{n_{k}M_{k}}\left( x\right)
A_{n^{\left( k\right) }-1}^{-\alpha }
\end{equation*}%
\begin{equation*}
+\sum\limits_{k=0}^{A}\sum\limits_{j=0}^{n_{k}M_{k}-1}\left\vert
A_{n^{\left( k-1\right) }+j}^{-\alpha -1}\right\vert \left\vert D_{j}\left(
x\right) \right\vert
\end{equation*}%
\begin{equation*}
:=B_{1}+B_{2}
\end{equation*}%
From (\ref{dir5}) and (\ref{A2}) we have 
\begin{equation}
B_{1}\leq c\left( \alpha \right) \sum\limits_{k=0}^{A}M_{k}^{-\alpha
}D_{M_{k}}\left( x\right) ,\text{ }n_{k}\neq 0.  \label{B1}
\end{equation}%
For $B_{2}$ we can apply (\ref{dir4})%
\begin{eqnarray*}
B_{2}
&=&\sum\limits_{k=0}^{A}\sum\limits_{r=0}^{n_{k}-1}\sum%
\limits_{j=0}^{M_{k}-1}\left\vert A_{n^{\left( k-1\right)
}+j+rM_{k}}^{-\alpha -1}\right\vert \left\vert D_{j+rM_{k}}\left( x\right)
\right\vert \\
&\leq
&\sum\limits_{k=0}^{A}\sum\limits_{r=0}^{n_{k}-1}\sum\limits_{j=0}^{M_{k}-1}%
\left\vert A_{n^{\left( k-1\right) }+j+rM_{k}}^{-\alpha -1}\right\vert
D_{M_{k}}\left( x\right) \\
&&+\sum\limits_{k=0}^{A}\sum\limits_{r=0}^{n_{k}-1}\sum%
\limits_{j=0}^{M_{k}-1}\left\vert A_{n^{\left( k-1\right)
}+j+rM_{k}}^{-\alpha -1}\right\vert \left\vert D_{j}\left( x\right)
\right\vert \\
&=&B_{21}+B_{22}.
\end{eqnarray*}%
When $r\neq 0$, from (\ref{A2}) we have 
\begin{equation*}
A_{n^{\left( k-1\right) }+j+rM_{k}}^{-\alpha -1}\sim M_{k}^{-\alpha -1},
\end{equation*}%
then

\begin{equation}
B_{21}\leq \sum\limits_{k=0}^{A}\sum\limits_{j=0}^{M_{k}-1}M_{k}^{-\alpha
-1}D_{M_{k}}\left( x\right) =\sum\limits_{k=0}^{A}M_{k}^{-\alpha
}D_{M_{k}}\left( x\right) .  \label{B21}
\end{equation}%
It is known that \cite{SWSP}

\begin{equation}
D_{n}\left( x\right) =\psi _{n}\left( x\right) \sum\limits_{j=0}^{\infty
}D_{M_{j}}\left( x\right) \sum\limits_{a=m_{j}-n_{j}}^{m_{j}-1}r_{j}^{a}.
\label{Dn}
\end{equation}%
Suppose that in $n^{\left( k-1\right) }$ $\ n_{k-1}=n_{k-2}=\cdots
=n_{t+1}=0 $ and $n_{t}\neq 0.$ Then%
\begin{eqnarray*}
B_{22} &\leq
&\sum\limits_{k=0}^{A}\sum\limits_{q=0}^{t}\sum\limits_{j=M_{q}}^{M_{q+1}-1}%
\left\vert A_{n^{\left( k-1\right) }+j}^{-\alpha -1}\right\vert \left\vert
D_{j}\left( x\right) \right\vert \\
&&+\sum\limits_{k=0}^{A}\sum\limits_{q=t+1}^{k-1}\sum%
\limits_{j=M_{q}}^{M_{q+1}-1}\left\vert A_{n^{\left( k-1\right)
}+j}^{-\alpha -1}\right\vert \left\vert D_{j}\left( x\right) \right\vert \\
&=&B_{22}^{^{\prime }}+B_{22}^{^{\prime \prime }}.
\end{eqnarray*}

We can write%
\begin{equation*}
\sum\limits_{q=0}^{t}\sum\limits_{j=M_{q}}^{M_{q+1}-1}\left\vert
A_{n^{\left( k-1\right) }+j}^{-\alpha -1}\right\vert \left\vert D_{j}\left(
x\right) \right\vert \leq
\sum\limits_{q=0}^{t}\sum\limits_{j=M_{q}}^{M_{q+1}-1}M_{t}^{-\alpha
-1}\sum\limits_{l=0}^{q}D_{M_{l}}\left( x\right)
\end{equation*}%
\begin{equation*}
\leq \sum\limits_{q=0}^{t}\left( M_{q+1}-M_{q}\right) M_{t}^{-\alpha
-1}\sum\limits_{l=0}^{q}D_{M_{l}}\left( x\right)
\end{equation*}%
\begin{equation*}
\leq \sum\limits_{q=0}^{t}\left( M_{q+1}-M_{q}\right) M_{t}^{-\alpha
-1}\sum\limits_{l=0}^{t}D_{M_{l}}\left( x\right)
\end{equation*}%
\begin{equation*}
\leq M_{t+1}M_{t}^{-\alpha -1}\sum\limits_{l=0}^{t}D_{M_{l}}\left( x\right)
\leq c\cdot M_{t}^{-\alpha }\sum\limits_{l=0}^{t}D_{M_{l}}\left( x\right) .
\end{equation*}

We get

\begin{equation}
B_{22}^{^{\prime }}\leq \sum\limits_{t=0}^{A}M_{t}^{-\alpha
}\sum\limits_{l=0}^{t}D_{M_{l}}\left( x\right) \leq
\sum\limits_{l=0}^{A}M_{l}^{-\alpha }D_{M_{l}}\left( x\right) .
\label{B22-1}
\end{equation}%
As we suppose $n_{k-1}=n_{k-2}=\cdots =n_{t+1}=0,$ from (\ref{Dn}) we can
write

\begin{equation*}
\left\vert D_{j}\left( x\right) \right\vert
=\sum\limits_{l=0}^{q}D_{M_{l}}\left( x\right)
=\sum\limits_{l=0}^{t}D_{M_{l}}\left( x\right) ,
\end{equation*}%
therefore%
\begin{eqnarray*}
\sum\limits_{q=t+1}^{k-1}\sum\limits_{j=M_{q}}^{M_{q+1}-1}\left\vert
A_{n^{\left( k-1\right) }+j}^{-\alpha -1}\right\vert \left\vert D_{j}\left(
x\right) \right\vert &\leq
&\sum\limits_{q=t+1}^{k-1}\sum\limits_{j=M_{q}}^{M_{q+1}-1}M_{q}^{-\alpha
-1}\sum\limits_{l=0}^{q}D_{M_{l}}\left( x\right) \\
&\leq &\sum\limits_{q=t+1}^{k-1}M_{q}^{-\alpha
}\sum\limits_{l=0}^{q}D_{M_{l}}\left( x\right) \leq M_{t}^{-\alpha
}\sum\limits_{l=0}^{t}D_{M_{l}}\left( x\right) .
\end{eqnarray*}%
Consequently,%
\begin{equation}
B_{22}^{^{\prime \prime }}\leq c\left( \alpha \right)
\sum\limits_{l=0}^{A}M_{l}^{-\alpha }D_{M_{l}}\left( x\right) .
\label{B22-2}
\end{equation}%
Combining (\ref{B1}), (\ref{B21}), (\ref{B22-1}) and (\ref{B22-2}) \ we
complete the proof of Lemma \ref{Lemma2}.
\end{proof}

For Walsh system Lemma \ref{Lemma1} and Lemma \ref{Lemma2} see in \cite{gog1}%
.

Suppose that $Z_{\beta }^{\left( k\right) }=\left( 0,\ldots ,0,x_{q}\neq
0,x_{q+1},\ldots ,x_{k-1},0,\ldots \right) ,$ then 
\begin{equation*}
\beta =\sum\limits_{j=q}^{k-1}\left( \frac{x_{j}}{M_{j+1}}\right) M_{k}\sim 
\frac{M_{k}}{M_{q}}
\end{equation*}

\begin{lemma}
\label{Lemma3} Let $\alpha \in \left( 0,1\right) .$ Then%
\begin{equation*}
\left\vert K_{n}^{-\alpha }\left( Z_{\beta }^{\left( k\right) }\right)
\right\vert \leq \frac{c\left( \alpha \right) }{\beta ^{1-\alpha }}M_{k}.
\end{equation*}
\end{lemma}

\begin{proof}
From Lemma \ref{Lemma2}, (\ref{A2}) \ and (\ref{dir1}) we get\bigskip 
\begin{eqnarray*}
\left\vert K_{n}^{-\alpha }\left( Z_{\beta }^{\left( k\right) }\right)
\right\vert &\leq &\frac{c\left( \alpha \right) }{A_{n-1}^{-\alpha }}%
\sum\limits_{l=0}^{A}M_{l}^{-\alpha }D_{M_{l}}\left( Z_{\beta }^{\left(
k\right) }\right) \\
&\leq &\frac{c\left( \alpha \right) }{A_{n-1}^{-\alpha }}\sum%
\limits_{l=0}^{q}M_{l}^{-\alpha }M_{l}\leq \frac{c\left( \alpha \right) }{%
\beta ^{1-\alpha }}M_{k}.
\end{eqnarray*}
\end{proof}

\begin{lemma}
\label{agaev} \cite{AVDR} Let $a_{1},\ldots ,a_{n}$ be real numbers. Then 
\begin{equation*}
\frac{1}{n}\int\limits_{G_{m}}\left\vert
\sum\limits_{k=1}^{n}a_{k}D_{k}\left( x\right) \right\vert d\mu \left(
x\right) \leq \frac{c}{\sqrt{n}}\left( \sum\limits_{k=1}^{n}a_{k}^{2}\right)
^{\frac{1}{2}},
\end{equation*}%
where c is absolute constant.
\end{lemma}

\begin{lemma}
\cite{Tep} \label{tsitsi} \bigskip Let $f\in C\left( G_{m}\right) .$ \ Then
for every $\alpha \in \left( 0,1\right) $ the following estimation holds%
\begin{equation*}
\frac{1}{A_{n}^{-\alpha }}\left\Vert \int\limits_{G_{m}}\sum\limits_{\nu
=0}^{M_{k-1}-1}A_{n-\nu }^{-\alpha }\psi _{\nu }\left( u\right) \left[
f\left( \cdot +u\right) -f\left( \cdot \right) \right] d\mu \left( u\right)
\right\Vert _{C}
\end{equation*}%
\begin{equation*}
\leq c\left( p,\alpha \right) \sum\limits_{r=0}^{k-1}\frac{M_{r}}{M_{k}}%
\omega \left( \frac{1}{M_{k}},f\right) _{C},
\end{equation*}%
where $M_{k}\leq n<M_{k+1}.$
\end{lemma}

\section{Proofs of main results}

\begin{proof}
of Theorem \ref{theorem1}. From (\ref{dir1}) we can write 
\begin{eqnarray*}
\sigma _{n}^{-\alpha }(f;x)-f\left( x\right) &=&\frac{1}{A_{n-1}^{-\alpha }}%
\int\limits_{G_{m}}\sum\limits_{\upsilon =0}^{n-1}A_{n-1-\upsilon }^{-\alpha
}\psi _{\upsilon }\left( t\right) \left[ f\left( x-t\right) -f\left(
x\right) \right] d\mu \left( t\right) \\
&=&\frac{1}{A_{n-1}^{-\alpha }}\int\limits_{G_{m}}\sum\limits_{\upsilon
=0}^{M_{k-1}-1}A_{n-1-\upsilon }^{-\alpha }\psi _{\upsilon }\left( t\right) %
\left[ f\left( x-t\right) -f\left( x\right) \right] d\mu \left( t\right) \\
&&+\frac{1}{A_{n-1}^{-\alpha }}\int\limits_{G_{m}}\sum\limits_{\upsilon
=M_{k-1}}^{M_{k}-1}A_{n-1-\upsilon }^{-\alpha }\psi _{\upsilon }\left(
t\right) \left[ f\left( x-t\right) -f\left( x\right) \right] d\mu \left(
t\right) \\
&&+\frac{1}{A_{n-1}^{-\alpha }}\int\limits_{G_{m}}\sum\limits_{\upsilon
=M_{k}}^{n_{k}M_{k}-1}A_{n-1-\upsilon }^{-\alpha }\psi _{\upsilon }\left(
t\right) \left[ f\left( x-t\right) -f\left( x\right) \right] d\mu \left(
t\right) \\
&&+\frac{1}{A_{n-1}^{-\alpha }}\int\limits_{G_{m}}\sum\limits_{\upsilon
=n_{k}M_{k}}^{n-1}A_{n-1-\upsilon }^{-\alpha }\psi _{\upsilon }\left(
t\right) \left[ f\left( x-t\right) -f\left( x\right) \right] d\mu \left(
t\right) \\
&=&I+II+III+IV.
\end{eqnarray*}%
\ \ First we estimate $IV$, we have%
\begin{equation*}
IV=\frac{A_{n^{^{\prime }}-1}^{-\alpha }}{A_{n-1}^{-\alpha }}%
\int\limits_{G_{m}}\psi _{M_{k}}^{n_{k}}\left( t\right) K_{n^{^{\prime
}}}^{-\alpha }\left( t\right) \left[ f\left( x-t\right) -f\left( x\right) %
\right] d\mu \left( t\right) .
\end{equation*}%
Since 
\begin{equation*}
D_{n^{^{\prime }}}\left( Z_{\beta }^{\left( k\right) }+t\right)
=D_{n^{^{\prime }}}\left( Z_{\beta }^{\left( k\right) }\right) ,\text{ \ \ }%
x\in I_{k},\text{ }0\leq \beta <M_{k},\text{ }n^{^{\prime }}<M_{k},
\end{equation*}%
we have 
\begin{equation}
K_{n^{^{\prime }}}^{-\alpha }\left( Z_{\beta }^{\left( k\right) }+t\right)
=K_{n^{^{\prime }}}^{-\alpha }\left( Z_{\beta }^{\left( k\right) }\right) .
\label{Kn}
\end{equation}%
Hence, from (\ref{Gm}) we can write%
\begin{equation}
IV=\frac{A_{n^{^{\prime }}-1}^{-\alpha }}{A_{n-1}^{-\alpha }}%
\sum\limits_{\beta =0}^{M_{k}-1}\int\limits_{I_{k}+Z_{\beta }^{\left(
k\right) }}f\left( x-t\right) \psi _{M_{k}}^{n_{k}}\left( t\right)
K_{n^{^{\prime }}}^{-\alpha }\left( t\right) d\mu \left( t\right)  \notag
\end{equation}%
\begin{equation*}
=\frac{A_{n^{^{\prime }}-1}^{-\alpha }}{A_{n-1}^{-\alpha }}%
\sum\limits_{\beta =0}^{M_{k}-1}\int\limits_{I_{k}}f\left( x-t-Z_{\beta
}^{\left( k\right) }\right) \psi _{M_{k}}^{n_{k}}\left( t\right) \psi
_{M_{k}}^{n_{k}}\left( Z_{\beta }^{\left( k\right) }\right) K_{n^{^{\prime
}}}^{-\alpha }\left( Z_{\beta }^{\left( k\right) }\right) d\mu \left(
t\right)
\end{equation*}%
\begin{equation*}
=\frac{A_{n^{^{\prime }}-1}^{-\alpha }}{A_{n-1}^{-\alpha }}%
\int\limits_{I_{k}}f\left( x-t\right) \psi _{M_{k}}^{n_{k}}\left( t\right)
K_{n^{^{\prime }}}^{-\alpha }\left( 0\right) d\mu \left( t\right)
\end{equation*}%
\begin{equation*}
+\frac{A_{n^{^{\prime }}-1}^{-\alpha }}{A_{n-1}^{-\alpha }}%
\sum\limits_{\beta =1}^{M_{k}-1}\int\limits_{I_{k}}f\left( x-t-Z_{\beta
}^{\left( k\right) }\right) \psi _{M_{k}}^{n_{k}}\left( t\right) \psi
_{M_{k}}^{n_{k}}\left( Z_{\beta }^{\left( k\right) }\right) K_{n^{^{\prime
}}}^{-\alpha }\left( Z_{\beta }^{\left( k\right) }\right) d\mu \left(
t\right)
\end{equation*}%
\begin{equation*}
=IV_{1}+IV_{2}.
\end{equation*}%
It is clear that%
\begin{equation}
\psi _{M_{k}}^{-n_{k}}\left( e_{k}\right) \psi _{M_{k}}^{n_{k}}\left(
t\right) =e^{\frac{-2\pi i}{m_{k}}n_{k}}e^{\frac{2\pi it_{k}}{m_{k}}%
n_{k}}=e^{\frac{2\pi i\left( t_{k}-1\right) }{m_{k}}n_{k}}=\psi
_{M_{k}}^{n_{k}}\left( t-e_{k}\right)  \label{truk}
\end{equation}%
and%
\begin{eqnarray}
&&\left\vert 1-\psi _{M_{k}}^{-n_{k}}\left( e_{k}\right) \right\vert
\label{truk1} \\
&=&\left\vert 1-\cos \frac{2\pi }{m_{k}}n_{k}+i\sin \frac{2\pi }{m_{k}}%
n_{k}\right\vert  \notag \\
&=&\sqrt{1-2\cos \frac{2\pi }{m_{k}}n_{k}+\cos ^{2}\frac{2\pi }{m_{k}}%
n_{k}+\sin ^{2}\frac{2\pi }{m_{k}}n_{k}}  \notag \\
&=&\sqrt{2-2\cos \frac{2\pi }{m_{k}}n_{k}}=\left\vert 2\sin \frac{\pi }{m_{k}%
}n_{k}\right\vert \geq 2\sin \frac{\pi }{m}=c.  \notag
\end{eqnarray}%
We get%
\begin{eqnarray*}
\psi _{M_{k}}^{-n_{k}}\left( e_{k}\right) \cdot IV_{1} &=&\frac{%
A_{n^{^{\prime }}-1}^{-\alpha }}{A_{n-1}^{-\alpha }}\int\limits_{I_{k}}f%
\left( x-t\right) K_{n^{^{\prime }}}^{-\alpha }\left( 0\right) \psi
_{M_{k}}^{n_{k}}\left( t-e_{k}\right) \psi _{M_{k}}^{n_{k}}\left( 0\right)
d\mu \left( t\right) \\
&=&\frac{A_{n^{^{\prime }}-1}^{-\alpha }}{A_{n-1}^{-\alpha }}%
\int\limits_{I_{k}}f\left( x-t-e_{k}\right) K_{n^{^{\prime }}}^{-\alpha
}\left( 0\right) \psi _{M_{k}}^{n_{k}}\left( t\right) \psi
_{M_{k}}^{n_{k}}\left( 0\right) d\mu \left( t\right)
\end{eqnarray*}%
Since $\left\vert K_{n}^{-\alpha }\left( f\right) \right\vert =O\left(
n\right) $ we have%
\begin{equation}
\left\vert IV_{1}-\psi _{M_{k}}^{-n_{k}}\left( e_{k}\right) IV_{1}\right\vert
\notag
\end{equation}%
\begin{equation*}
\leq \left\vert IV_{1}\right\vert \cdot \left\vert 1-\psi
_{M_{k}}^{-n_{k}}\left( e_{k}\right) \right\vert
\end{equation*}%
\begin{equation*}
\leq \left\vert 1-\psi _{M_{k}}^{-n_{k}}\left( e_{k}\right) \right\vert 
\frac{A_{n^{^{\prime }}-1}^{-\alpha }}{A_{n-1}^{-\alpha }}%
\int\limits_{I_{k}}\left\vert \left( f\left( x-t\right) -f\left(
x-t-e_{k}\right) \right) K_{n^{^{\prime }}}^{-\alpha }\left( 0\right) \psi
_{M_{k}}^{n_{k}}\left( t\right) \right\vert d\mu \left( t\right)
\end{equation*}%
\begin{equation*}
\leq c\left( \alpha \right) \omega \left( f,\frac{1}{M_{k}}\right) .
\end{equation*}

Hence from (\ref{truk1})%
\begin{equation}
\left\vert IV_{1}\right\vert \leq c\left( \alpha \right) \omega \left( f,%
\frac{1}{M_{k}}\right) .  \label{IV1}
\end{equation}%
Analogously, from the condition of theorem and Lemma \ref{Lemma3} we can
write%
\begin{equation*}
\left\vert IV_{2}-\psi _{M_{k}}^{-n_{k}}\left( e_{k}\right) IV_{2}\right\vert
\end{equation*}%
\begin{equation*}
\leq \left\vert IV_{2}\right\vert \cdot \left\vert 1-\psi
_{M_{k}}^{-n_{k}}\left( e_{k}\right) \right\vert
\end{equation*}%
\begin{eqnarray*}
&\leq &\left\vert 1-\psi _{M_{k}}^{-n_{k}}\left( e_{k}\right) \right\vert 
\frac{A_{n^{^{\prime }}-1}^{-\alpha }}{A_{n-1}^{-\alpha }}%
M_{k}\int\limits_{I_{k}}\sum\limits_{\beta =1}^{M_{k}-1}\frac{1}{\beta
^{1-\alpha }} \\
&&\times \left\vert \left( f\left( x-t-Z_{\beta }^{\left( k\right) }\right)
-f\left( x-t-Z_{\beta }^{\left( k\right) }-e_{k}\right) \right) \psi
_{M_{k}}^{n_{k}}\left( t\right) \right\vert d\mu \left( t\right) \\
&\leq &c\left( \alpha \right) \frac{\left( n^{^{\prime }}\right) ^{-\alpha }%
}{M_{k}^{1-\alpha }}M_{k}^{1-\alpha }o\left( 1\right) =o\left( 1\right) ,
\end{eqnarray*}%
therefore from (\ref{truk1}) 
\begin{equation}
IV_{2}=o\left( 1\right) ,\text{ \ }as\text{ }k\rightarrow \infty \text{ }
\label{IV2}
\end{equation}%
uniformly with respect to $x\in G_{m}.$

combining (\ref{IV1}) and (\ref{IV2}) we conclude that 
\begin{equation}
IV=o\left( 1\right) ,\text{ \ }as\text{ }k\rightarrow \infty \text{ }
\label{IV}
\end{equation}%
uniformly with respect to $x\in G_{m}.$

For $III$ we can write 
\begin{equation}
III=\frac{1}{A_{n-1}^{-\alpha }}\int\limits_{G_{m}}\sum\limits_{\upsilon
=M_{k}}^{\left( n_{k}-1\right) M_{k}-1}A_{n-1-\upsilon }^{-\alpha }\psi
_{\upsilon }\left( t\right) \left[ f\left( x-t\right) -f\left( x\right) %
\right] d\mu \left( t\right)  \label{III}
\end{equation}%
\begin{equation*}
+\frac{1}{A_{n-1}^{-\alpha }}\int\limits_{G_{m}}\sum\limits_{\upsilon
=\left( n_{k}-1\right) M_{k}}^{n_{k}M_{k}-1}A_{n-1-\upsilon }^{-\alpha }\psi
_{\upsilon }\left( t\right) \left[ f\left( x-t\right) -f\left( x\right) %
\right] d\mu \left( t\right)
\end{equation*}%
\begin{equation*}
=III_{1}+III_{2}.
\end{equation*}%
Since 
\begin{equation*}
\left\Vert f-S_{M_{k}}\left( f\right) \right\Vert _{C}\leq c\cdot \omega
\left( f,\frac{1}{M_{k}}\right) _{C}
\end{equation*}%
Applying Abel's transformation, from ortogonality of Vilenkin system and (%
\ref{A2}) we get%
\begin{equation*}
\left\vert III_{1}\right\vert \leq \frac{1}{A_{n-1}^{-\alpha }}%
\int\limits_{G_{m}}\left\vert \sum\limits_{\upsilon =M_{k}}^{\left(
n_{k}-1\right) M_{k}-1}A_{n-1-\upsilon }^{-\alpha }\psi _{\upsilon }\left(
t\right) \left[ f\left( x-t\right) -S_{M_{k}}\left( f,x-t\right) \right]
\right\vert d\mu \left( t\right)
\end{equation*}%
\begin{equation*}
\leq \frac{1}{A_{n-1}^{-\alpha }}\omega \left( f,\frac{1}{M_{k}}\right)
_{C}\int\limits_{G_{m}}\left\vert \sum\limits_{\upsilon =M_{k}}^{\left(
n_{k}-1\right) M_{k}-1}A_{n-1-\upsilon }^{-\alpha }\psi _{\upsilon }\left(
t\right) \right\vert d\mu \left( t\right)
\end{equation*}%
\begin{equation*}
\leq c\left( \alpha \right) n^{\alpha }\omega \left( f,\frac{1}{M_{k}}%
\right) _{C}\int\limits_{G_{m}}\left\vert \sum\limits_{\upsilon
=M_{k}}^{\left( n_{k}-1\right) M_{k}-2}A_{n-1-\upsilon }^{-\alpha
-1}D_{\upsilon }\left( t\right) \right\vert d\mu \left( t\right)
\end{equation*}%
\begin{equation*}
+c\left( \alpha \right) n^{\alpha }\omega \left( f,\frac{1}{M_{k}}\right)
_{C}d\int\limits_{G_{m}}\left\vert A_{n-1-\left( n_{k}-1\right)
M_{k}-1}^{-\alpha }D_{\left( n_{k}-1\right) M_{k}-1}\left( t\right)
\right\vert \mu \left( t\right)
\end{equation*}%
\begin{equation*}
+c\left( \alpha \right) n^{\alpha }\omega \left( f,\frac{1}{M_{k}}\right)
_{C}\int\limits_{G_{m}}A_{n-1-M_{k}}^{-\alpha }D_{M_{k}}\left( t\right) d\mu
\left( t\right)
\end{equation*}%
\begin{equation*}
=III_{11}+III_{12}+III_{13}.
\end{equation*}%
From Lemma \ref{agaev} and (\ref{A2}) we get%
\begin{eqnarray}
III_{11} &\leq &c\sqrt{n}n^{\alpha }\left( \sum\limits_{\upsilon
=M_{k}}^{\left( n_{k}-1\right) M_{k}-2}\left( n-1-\upsilon \right)
^{-2\alpha -2}\right) ^{\frac{1}{2}}\omega \left( f,\frac{1}{M_{k}}\right)
_{C}  \label{11} \\
&\leq &c\left( \alpha \right) \omega \left( f,\frac{1}{M_{k}}\right) _{C}. 
\notag
\end{eqnarray}%
Consequently, from (\ref{dir}) \ and \ (\ref{dir5}) we get%
\begin{equation}
III_{12}\leq c\left( \alpha \right) \omega \left( f,\frac{1}{M_{k}}\right)
_{C}.  \label{12}
\end{equation}%
Analogously,%
\begin{equation}
III_{13}\leq c\left( \alpha \right) \omega \left( f,\frac{1}{M_{k}}\right)
_{C}.  \label{13}
\end{equation}

Combining (\ref{11})- (\ref{13}) we obtain%
\begin{equation}
\left\vert III_{1}\right\vert \leq c\left( \alpha \right) \omega \left( f,%
\frac{1}{M_{k}}\right) _{C}.  \label{III1}
\end{equation}

Now, we estimate $III_{2}.$ Let $n_{k}>1.$ It is clear that%
\begin{equation*}
\sum\limits_{\upsilon =\left( n_{k}-1\right)
M_{k}}^{n_{k}M_{k}-1}A_{n-1-\upsilon }^{-\alpha }\psi _{\upsilon }\left(
t\right) =\sum\limits_{\upsilon =0}^{M_{k}-1}A_{n-1-\left( n_{k}-1\right)
M_{k}-\upsilon }^{-\alpha }\psi _{\upsilon +\left( n_{k}-1\right)
M_{k}}\left( t\right)
\end{equation*}%
\begin{equation*}
=\psi _{M_{k}}^{n_{k}-1}\left( t\right) \sum\limits_{\upsilon
=0}^{M_{k}-1}A_{n-1-\left( n_{k}-1\right) M_{k}-\upsilon }^{-\alpha }\psi
_{\upsilon }\left( t\right)
\end{equation*}%
\begin{equation*}
=\psi _{M_{k}}^{n_{k}-1}\left( t\right) \sum\limits_{\upsilon
=0}^{n-1-\left( n_{k}-1\right) M_{k}}A_{n-1-\left( n_{k}-1\right)
M_{k}-\upsilon }^{-\alpha }\psi _{\upsilon }\left( t\right)
\end{equation*}%
\begin{equation*}
-\psi _{M_{k}}^{n_{k}-1}\left( t\right) \sum\limits_{\upsilon
=M_{k}}^{n-1-\left( n_{k}-1\right) M_{k}}A_{n-1-\left( n_{k}-1\right)
M_{k}-\upsilon }^{-\alpha }\psi _{\upsilon }\left( t\right)
\end{equation*}%
\begin{equation*}
=\psi _{M_{k}}^{n_{k}-1}\left( t\right) \sum\limits_{\upsilon
=0}^{n-1-\left( n_{k}-1\right) M_{k}}A_{n-1-\left( n_{k}-1\right)
M_{k}-\upsilon }^{-\alpha }\psi _{\upsilon }\left( t\right)
\end{equation*}%
\begin{equation*}
-\psi _{M_{k}}^{n_{k}}\left( t\right) \sum\limits_{\upsilon
=0}^{n-1-n_{k}M_{k}}A_{n-1-n_{k}M_{k}-\upsilon }^{-\alpha }\psi _{\upsilon
}\left( t\right)
\end{equation*}%
\begin{equation*}
=A_{n-1-\left( n_{k}-1\right) M_{k}}^{-\alpha }\psi _{M_{k}}^{n_{k}-1}\left(
t\right) K_{n-1-\left( n_{k}-1\right) M_{k}}^{-\alpha }\left( t\right)
\end{equation*}%
\begin{equation*}
-A_{n-1-n_{k}M_{k}}^{-\alpha }\psi _{M_{k}}^{n_{k}}\left( t\right)
K_{n-1-n_{k}M_{k}}^{-\alpha }\left( t\right) .
\end{equation*}%
Hence, 
\begin{eqnarray}
III_{2} &=&\frac{1}{A_{n-1}^{-\alpha }}\int\limits_{G_{m}}A_{n-1-\left(
n_{k}-1\right) M_{k}}^{-\alpha }K_{n-1-\left( n_{k}-1\right) M_{k}}^{-\alpha
}\left( t\right)  \label{III2} \\
&&\times \psi _{M_{k}}^{n_{k}-1}\left( t\right) \left[ f\left( x-t\right)
-f\left( x\right) \right] d\mu \left( t\right)  \notag \\
&&-\frac{1}{A_{n-1}^{-\alpha }}\int\limits_{G_{m}}A_{n-1-n_{k}M_{k}}^{-%
\alpha }K_{n-1-n_{k}M_{k}}^{-\alpha }\left( t\right)  \notag \\
&&\times \psi _{M_{k}}^{n_{k}}\left( t\right) \left[ f\left( x-t\right)
-f\left( x\right) \right] d\mu t  \notag
\end{eqnarray}%
\begin{equation*}
=III_{21}+III_{22.}
\end{equation*}%
From (\ref{Gm}) we can write%
\begin{equation}
III_{21}=\frac{A_{n-1-\left( n_{k}-1\right) M_{k}}^{-\alpha }}{%
A_{n-1}^{-\alpha }}\sum\limits_{\beta
=0}^{M_{k}-1}\int\limits_{I_{k}+Z_{\beta }^{\left( k\right) }}f\left(
x-t\right)  \label{III2'}
\end{equation}%
\begin{equation*}
\times K_{n-1-\left( n_{k}-1\right) M_{k}}^{-\alpha }\left( t\right) \psi
_{M_{k}}^{n_{k}-1}\left( t\right) d\mu \left( t\right)
\end{equation*}%
\begin{equation*}
=\frac{A_{n-1-\left( n_{k}-1\right) M_{k}}^{-\alpha }}{A_{n-1}^{-\alpha }}%
\sum\limits_{\beta =0}^{M_{k}-1}\int\limits_{I_{k}}f\left( x-t-Z_{\beta
}^{\left( k\right) }\right)
\end{equation*}%
\begin{equation*}
\times K_{n-1-\left( n_{k}-1\right) M_{k}}^{-\alpha }\left( Z_{\beta
}^{\left( k\right) }+t\right) \psi _{M_{k}}^{n_{k}-1}\left( Z_{\beta
}^{\left( k\right) }+t\right) d\mu \left( t\right)
\end{equation*}%
\begin{equation*}
=\frac{A_{n-1-\left( n_{k}-1\right) M_{k}}^{-\alpha }}{A_{n-1}^{-\alpha }}%
\int\limits_{I_{k}}f\left( x-t\right) K_{n-1-\left( n_{k}-1\right)
M_{k}}^{-\alpha }\left( 0\right) \psi _{M_{k}}^{n_{k}-1}\left( t\right) d\mu
\left( t\right)
\end{equation*}%
\begin{equation*}
+\frac{A_{n-1-\left( n_{k}-1\right) M_{k}}^{-\alpha }}{A_{n-1}^{-\alpha }}%
\sum\limits_{\beta =1}^{M_{k}-1}\int\limits_{I_{k}}f\left( x-t-Z_{\beta
}^{\left( k\right) }\right)
\end{equation*}%
\begin{equation*}
\times K_{n-1-\left( n_{k}-1\right) M_{k}}^{-\alpha }\left( Z_{\beta
}^{\left( k\right) }\right) \psi _{M_{k}}^{n_{k}-1}\left( t\right) \psi
_{M_{k}}^{n_{k}-1}\left( Z_{\beta }^{\left( k\right) }\right) d\mu \left(
t\right)
\end{equation*}%
\begin{equation*}
=III_{211}+III_{212}.
\end{equation*}%
Since (see (\ref{truk} ) and (\ref{truk1}))%
\begin{equation*}
\psi _{M_{k}}^{-\left( n_{k}-1\right) }\left( e_{k}\right) \psi
_{M_{k}}^{n_{k}-1}\left( t\right) =\psi _{M_{k}}^{n_{k}-1}\left(
t-e_{k}\right)
\end{equation*}%
and%
\begin{equation}
\left\vert 1-\psi _{M_{k}}^{-\left( n_{k}-1\right) }\left( e_{k}\right)
\right\vert \geq c>0.  \label{c}
\end{equation}%
We get%
\begin{equation*}
\left\vert III_{211}-\psi _{M_{k}}^{-\left( n_{k}-1\right) }\left(
e_{k}\right) III_{211}\right\vert
\end{equation*}%
\begin{equation*}
\leq \left\vert III_{211}\right\vert \cdot \left\vert 1-\psi
_{M_{k}}^{-\left( n_{k}-1\right) }\left( e_{k}\right) \right\vert
\end{equation*}%
\begin{eqnarray*}
&\leq &\frac{A_{n-1-\left( n_{k}-1\right) M_{k}}^{-\alpha }}{%
A_{n-1}^{-\alpha }}\left\vert 1-\psi _{M_{k}}^{-\left( n_{k}-1\right)
}\left( e_{k}\right) \right\vert \\
&&\times \int\limits_{I_{k}}\left\vert \left( f\left( x-t\right) -f\left(
x-t-e_{k}\right) \right) K_{n-1-\left( n_{k}-1\right) M_{k}}^{-\alpha
}\left( 0\right) \psi _{M_{k}}^{n_{k}-1}\left( t\right) \right\vert d\mu
\left( t\right) \\
&\leq &c\left( \alpha \right) \omega \left( f,\frac{1}{M_{k}}\right) _{C}.
\end{eqnarray*}%
Since (\ref{c}) we have 
\begin{equation}
\left\vert III_{211}\right\vert \leq c\left( \alpha \right) \omega \left( f,%
\frac{1}{M_{k}}\right) _{C}.  \label{III211}
\end{equation}%
Analogously for $III_{212}$ from the condition of theorem and Lemma \ref%
{Lemma3} we get%
\begin{equation*}
\left\vert III_{212}-\psi _{M_{k}}^{-\left( n_{k}-1\right) }\left(
e_{k}\right) III_{212}\right\vert
\end{equation*}%
\begin{equation*}
\leq \frac{A_{n-1-\left( n_{k}-1\right) M_{k}}^{-\alpha }}{A_{n-1}^{-\alpha }%
}\left\vert 1-\psi _{M_{k}}^{-\left( n_{k}-1\right) }\left( e_{k}\right)
\right\vert M_{k}
\end{equation*}%
\begin{equation*}
\times \int\limits_{I_{k}}\sum\limits_{\beta =1}^{M_{k}-1}\frac{1}{\beta
^{1-\alpha }}\left\vert \left( f\left( x-t-Z_{\beta }^{\left( k\right)
}\right) -f\left( x-t-Z_{\beta }^{\left( k\right) }-e_{k}\right) \right)
\psi _{M_{k}}^{n_{k}-1}\left( t\right) \right\vert d\mu \left( t\right)
\end{equation*}%
\begin{equation*}
\leq c\left( \alpha \right) \frac{M_{k}^{-\alpha }}{M_{k}^{1-\alpha }}%
M_{k}^{1-\alpha }o\left( 1\right) =o(1),
\end{equation*}%
and consequently, 
\begin{equation}
\left\vert III_{212}\right\vert =o(1),\text{ \ }as\text{ }k\rightarrow \infty
\label{III212}
\end{equation}%
uniformly with respect to $x\in G_{m}.$ Combining (\ref{III2'}), (\ref%
{III211}) and (\ref{III212}) we have 
\begin{equation*}
\left\vert III_{21}\right\vert =o(1),\text{ \ }as\text{ }k\rightarrow \infty
\end{equation*}%
uniformly with respect to $x\in G_{m}.$

Analogously, we can prove the estimation for $III_{22}$%
\begin{equation*}
\left\vert III_{22}\right\vert =o(1),\text{ \ }as\text{ }k\rightarrow \infty
\end{equation*}%
uniformly with respect to $x\in G_{m}.$

Finally for $III$ we obtain 
\begin{equation}
\left\vert III\right\vert =o(1),\text{ \ }as\text{ }k\rightarrow \infty
\label{3}
\end{equation}%
uniformly with respect to $x\in G_{m}.$

The estimation of $II$ is analogous to the estimation of $III$ and we can
conclude 
\begin{equation}
\left\vert II\right\vert =o(1),\text{ \ }as\text{ }k\rightarrow \infty
\label{II}
\end{equation}%
uniformly with respect to $x\in G_{m}.$

Using lemma \ref{tsitsi} we obtain%
\begin{equation}
I=o\left( 1\right) ,\text{ \ }as\text{ }k\rightarrow \infty  \label{I}
\end{equation}%
uniformly with respect to $x\in G_{m}.$

Combining (\ref{IV}), (\ref{3}), (\ref{II}) and (\ref{I}) we complete the
proof of theorem \ref{theorem1}.
\end{proof}

\begin{proof}
of theorem \ref{theorem2}. \ It is clear that 
\begin{eqnarray*}
&&\sum\limits_{\beta =1}^{M_{k}-1}\frac{1}{\beta ^{1-\alpha }}\left\vert
f\left( x-Z_{\beta }^{\left( k\right) }\right) -f\left( x-Z_{\beta }^{\left(
k\right) }-e_{k}\right) \right\vert \\
&=&\sum\limits_{r=0}^{k-1}\sum\limits_{\beta =M_{r}}^{M_{r+1}-1}\frac{1}{%
\beta ^{1-\alpha }}\left\vert f\left( x-Z_{\beta }^{\left( k\right) }\right)
-f\left( x-Z_{\beta }^{\left( k\right) }-e_{k}\right) \right\vert \\
&\leq &\sum\limits_{r=0}^{k-1}\frac{1}{M_{r}^{1-\alpha }}\sum\limits_{\beta
=M_{r}}^{M_{r+1}-1}\left\vert f\left( x-Z_{\beta }^{\left( k\right) }\right)
-f\left( x-Z_{\beta }^{\left( k\right) }-e_{k}\right) \right\vert \\
&=&\sum\limits_{r=0}^{\gamma \left( k\right) }\frac{1}{M_{r}^{1-\alpha }}%
\sum\limits_{\beta =M_{r}}^{M_{r+1}-1}\left\vert f\left( x-Z_{\beta
}^{\left( k\right) }\right) -f\left( x-Z_{\beta }^{\left( k\right)
}-e_{k}\right) \right\vert \\
&&+\sum\limits_{r=\gamma \left( k\right) }^{k-1}\frac{1}{M_{r}^{1-\alpha }}%
\sum\limits_{\beta =M_{r}}^{M_{r+1}-1}\left\vert f\left( x-Z_{\beta
}^{\left( k\right) }\right) -f\left( x-Z_{\beta }^{\left( k\right)
}-e_{k}\right) \right\vert \\
&\leq &\omega \left( f,\frac{1}{M_{k}}\right) \sum\limits_{r=0}^{\gamma
\left( k\right) }M_{r}^{\alpha }+\sum\limits_{r=\gamma \left( k\right)
}^{k-1}\frac{1}{M_{r}^{1-\alpha }}\upsilon \left( M_{r},f\right) \\
&\leq &\omega \left( f,\frac{1}{M_{k}}\right) M_{\gamma \left( k\right)
}^{\alpha }+\sum\limits_{r=\gamma \left( k\right) }^{k-1}\frac{1}{%
M_{r}^{1-\alpha }}\upsilon \left( M_{r},f\right)
\end{eqnarray*}%
It is not hard to see that there exists $\left\{ \gamma \left( k\right)
:k\geq 1\right\} $ sequence for which $\gamma \left( k\right) \rightarrow
\infty ,$ as $k\rightarrow \infty $ and \ $\omega \left( f,\frac{1}{M_{k}}%
\right) M_{\gamma \left( k\right) }^{\alpha }\rightarrow 0.$ Hence, from the
condition of theorem \ we conclude that $\sum\limits_{r=\gamma \left(
k\right) }^{k-1}\frac{1}{M_{r}^{1-\alpha }}\upsilon \left( M_{r},f\right)
\rightarrow 0$, as $k\rightarrow \infty .$

Theorem \ref{theorem2} is proved.
\end{proof}

\begin{proof}
of corollary \ref{shedegi1}. \ 

Set 
\begin{equation*}
\sup_{k}\sum\limits_{\beta =1}^{M_{k}-1}M\left( \omega \left(
f,I_{k}+Z_{\beta }^{\left( k\right) }\right) \right) \equiv O_{M}.
\end{equation*}%
Let $O_{M}<1$. Using Jensen's inequality we have%
\begin{eqnarray*}
&&M\left( \frac{1}{M_{k}}\sum\limits_{\beta =1}^{M_{k}-1}\omega \left(
f,I_{k}+Z_{\beta }^{\left( k\right) }\right) \right) \\
&\leq &\frac{1}{M_{k}}\sum\limits_{\beta =1}^{M_{k}-1}M\left( \omega \left(
f,I_{k}+Z_{\beta }^{\left( k\right) }\right) \right) \leq \frac{1}{M_{k}}%
O_{M}<\frac{1}{M_{k}}.
\end{eqnarray*}%
Consequantly, 
\begin{equation*}
\upsilon \left( M_{k},f\right) \leq M_{k}M^{-1}\left( \frac{1}{M_{k}}\right)
.
\end{equation*}%
and 
\begin{equation*}
\sum\limits_{k=1}^{\infty }\frac{\upsilon \left( M_{k},f\right) }{%
M_{k}^{1-\alpha }}\leq \sum\limits_{k=1}^{\infty }M_{k}^{\alpha
}M^{-1}\left( \frac{1}{M_{k}}\right) <\infty .
\end{equation*}%
Now, let $O_{M}>1$. Because function $M$ is convex, we get%
\begin{equation*}
M\left( \frac{\upsilon \left( M_{k},f\right) }{M_{k}\cdot O_{M}}\right) \leq 
\frac{1}{O_{M}}M\left( \frac{\upsilon \left( M_{k},f\right) }{M_{k}}\right)
\leq \frac{1}{O_{M}}\frac{1}{M_{k}}O_{M}=\frac{1}{M_{k}}.
\end{equation*}%
Hence, 
\begin{equation*}
\frac{\upsilon \left( M_{k},f\right) }{M_{k}}\leq O_{M}M^{-1}\left( \frac{1}{%
M_{k}}\right)
\end{equation*}%
and%
\begin{equation*}
\sum\limits_{k=1}^{\infty }\frac{\upsilon \left( M_{k},f\right) }{%
M_{k}^{1-\alpha }}\leq \sum\limits_{k=1}^{\infty }M_{k}^{\alpha
}M^{-1}\left( \frac{1}{M_{k}}\right) <\infty
\end{equation*}

From theorem (\ref{theorem2}) we obtain proof of corollary \ref{shedegi1}.
\end{proof}

\begin{proof}
of \ corollary \ref{shedegi2}. \ Let $M\left( u\right) =u^{p}$. Then%
\begin{equation*}
\sum\limits_{k=1}^{\infty }M_{k}^{\alpha }M^{-1}\left( \frac{1}{M_{k}}%
\right) =\sum\limits_{k=1}^{\infty }\frac{1}{M_{k}^{\alpha -\frac{1}{p}}}%
<\infty ,
\end{equation*}

when $\frac{1}{p}-\alpha >0.$

corollary \ref{shedegi2} is proved.
\end{proof}

\end{document}